\documentclass{amsart}

\usepackage[dvips]{graphicx}
\usepackage{amssymb,amsmath}
\usepackage[all]{xy}

\usepackage[dvipdfmx]{hyperref}

\newtheorem{theorem}{Theorem}[section]
\newtheorem{proposition}[theorem]{Proposition}
\newtheorem{lemma}[theorem]{Lemma}
\newtheorem{corollary}[theorem]{Corollary}

\theoremstyle{definition}
\newtheorem{definition}[theorem]{Definition}

\theoremstyle{remark}
\newtheorem{remark}[theorem]{Remark}

\numberwithin{equation}{section}

\newcommand{\tr}{\mathrm{tr}}

\newcommand{\ML}{\mathcal{ML}}

\newcommand{\SLC}{\mathrm{SL}(2,\mathbb{C})}
\newcommand{\PSLC}{\mathrm{PSL}(2,\mathbb{C})}

\title[Exotic components in linear slices]{Exotic components in linear slices of quasi-Fuchsian groups}
\author{Yuichi Kabaya}
\address{Faculty of Engineering, Kitami Institute of Technology,
165 Koen-cho Kitami, Hokkaido, JAPAN}
\email{kabaya@mail.kitami-it.ac.jp}
\thanks{The author is supported by JSPS Grant-in-Aid for Scientific Research No. 23244005 and Grant-in-Aid for Young Scientists (B) No. 26800038.}

\begin{document}
\maketitle

\begin{abstract}
The linear slice of quasi-Fuchsian once-punctured torus groups is defined by fixing the complex length of some simple closed curve to be a fixed positive real number. 
It is known that the linear slice is a union of disks, and it always has one standard component containing Fuchsian groups. 
Komori and Yamashita proved that there exist non-standard components if the length is sufficiently large. 
We give two other proofs of their theorem, one is based on some properties of length functions, 
and the other is based on the theory of complex projective structures and complex earthquakes.
From the latter proof, we can characterize the existence of non-standard components in terms of exotic projective structures with quasi-Fuchsian holonomy.
\end{abstract}

\section{Introduction}
The Ending Lamination Theorem for Kleinian surface groups proved by Brock-Canary-Minsky \cite{minsky_elt}, \cite{BCM} gives 
a complete classification of discrete faithful representations of surface groups into $\PSLC$. 
However, as surveyed in \cite{canary}, the set of these representations in the character variety is known to be quite complicated. 
In this paper, we study once-punctured torus groups, in particular, slices of the space of quasi-Fuchsian once-punctured torus groups, called \emph{linear slices}.

Let $S$ be a once-punctured torus.
We denote by $\mathcal{QF}(S)$ the space of conjugacy classes of faithful quasi-Fuchsian representations.
We fix an essential simple closed curve $\gamma$ and consider the complex length function 
$\lambda_\gamma : \mathcal{QF}(S) \to \mathbb{C} / 2 \pi \sqrt{-1} \mathbb{Z}$ (see \S \ref{subsec:complex_length}).
For a positive real number $l$, we define the \emph{linear slice} $\mathcal{QF}(l)$ by the subset of $\mathcal{QF}(S)$ satisfying 
the identity $\lambda_\gamma = l$.

By McMullen's disk convexity of $\mathcal{QF}(S)$ (see Theorem \ref{thm:mcmullen_disk_convexity}), it is known that $\mathcal{QF}(l)$ is a union of open (topological) disks.
We can easily observe that $\mathcal{QF}(l)$ always has a unique component containing Fuchsian representations.
This component is called the \emph{BM-slice} (Bers-Maskit slice) in \cite{keen-series04}, or the \emph{standard component} in \cite{komori-yamashita}.
Komori and Yamashita showed in \cite{komori-yamashita} that $\mathcal{QF}(l)$ coincides with this standard component if $l$ is sufficiently small based on Otal's work \cite{otal}.
They also showed the following.

\begin{theorem}[Komori-Yamashita \cite{komori-yamashita}]
\label{komori_yamashita_theorem}
If $l > 0$ is sufficiently large, there exist infinitely many components in the linear slice $\mathcal{QF}(l)$.
\end{theorem}
They did not show that there are infinitely many, but this is an easy consequence of the action of Dehn twists 
along the curve $\gamma$ (Corollary \ref{cor:infinitely_many_by_dehn_twists}).  
Their proof is based on an analysis of the Earle slice of quasi-Fuchsian once-punctured torus groups developed in \cite{komori-series}.
Since the Earle slice is defined by using some symmetry of the once-punctured torus, it is difficult to generalize their result to general hyperbolic surfaces.

In this paper we give two other proofs of Theorem \ref{komori_yamashita_theorem}. 
The first one is done using a theorem of Parker and Parkkonen \cite{parker-parkkonen} (see Theorem \ref{thm:parker-parkkonen}) 
and some explicit calculations of trace functions (\S \ref{subsec:trace_functions}).

The second proof provides more geometric information.
This is done using the theory of complex projective structures and complex earthquakes.
The linear slice $\mathcal{QF}(l)$ is lifted to the \emph{complex earthquake} \cite{mcmullen}, 
which is a subset of the space of marked complex projective structures.
Since each lift of a component of $\mathcal{QF}(l)$ belongs to the set of complex projective structures with quasi-Fuchsian holonomy, 
it is parametrized by Goldman's classification \cite{goldman} (see Theorem \ref{thm:goldman}).
We give a criterion for the existence of non-standard components in terms of Goldman's classification in Lemma \ref{lem:criterion_for_existence_of_non-standard_component}.
Roughly, a non-standard component of $\mathcal{QF}(l)$ corresponds to an \emph{exotic} component (see \S \ref{subsec:cpx_proj_str_with_qf_hol}, \cite{ito_LMS}, \cite{dumas}) of the set of complex projective structures with quasi-Fuchsian holonomy. 

This paper is organized as follows.
In \S \ref{sec:background}, we review the basics of Kleinian surface groups and the definition of linear slices.
In \S \ref{sec:character_variety}, we give an explicit parametrization of the linear slice and calculate trace functions in terms of them.
In \S \ref{sec:linear_slices}, we further investigate the linear slice, especially the BM-slice, then give the first proof.
In \S \ref{sec:projective_structure}, we review the theory of complex projective structures and provide the second proof.
As a complement, we discuss pleated surfaces associated to the linear slice in \S \ref{sec:pleated_surfaces}.

\noindent
{\bf Acknowledgments.}
I would like to thank Hideki Miyachi and Ken'ichi Ohshika for invaluable discussions during this work.
I also thank Yasushi Yamashita for sharing his program.
Finally, I would like to thank the referee for helpful comments.

\section{Background}
\label{sec:background}
We denote a surface of genus $g$ with $n$ punctures by $S_{g,n}$. 
We assume that the Euler characteristic $2-2g-n < 0$.
If $g$ and $n$ are clear from the context, we omit the subscripts.

\subsection{Quasi-Fuchsian groups}
\label{subsec:quasi-fuchsian}
Let $\Gamma$ be a discrete subgroup of $\PSLC$ isomorphic to $\pi_1(S)$.
Since $\PSLC$ can be identified with the orientation preserving isometry group of the 3-dimensional hyperbolic space $\mathbb{H}^3$, 
the quotient $\mathbb{H}^3/ \Gamma$ is a hyperbolic 3-manifold homotopy equivalent to $S$. 
The limit set $\Lambda(\Gamma)$ of $\Gamma$ is the accumulation points of the orbit $\Gamma \cdot p$ in the ideal boundary $\partial_{\infty} \mathbb{H}^3 \cong \mathbb{C}P^1$
for some $p \in \mathbb{H}^3$.
The domain of discontinuity $\Omega (\Gamma)$ is the complement of $\Lambda(\Gamma)$ in $\mathbb{C}P^1$.
If $\Omega (\Gamma)$ consists of exactly two open disks (resp. open round disks), $\Gamma$ is called \emph{quasi-Fuchsian} (resp. \emph{Fuchsian}).
Let $CH(\Gamma)$ be the convex hull of $\Lambda(\Gamma)$ in $\mathbb{H}^3$.
The \emph{convex core} $C(\Gamma)$ is the quotient $CH(\Gamma) / \Gamma$, which is also characterized as the smallest convex subset of $\mathbb{H}^3 / \Gamma$ 
provided that the inclusion map is homotopy equivalent.
If $\Gamma$ is quasi-Fuchsian but not Fuchsian, $\partial C(\Gamma)$ consists of two components.
We denote them by $\partial^{+} C(\Gamma)$ and $\partial^{-} C(\Gamma)$. 
If $\Gamma$ is Fuchsian, $C(\Gamma)$ is a totally geodesic surface in $\mathbb{H}^3 / \Gamma$. 
In this case, we define $\partial^{\pm} C(\Gamma)$ by this surface.
It is known that the induced path metric on $\partial^{\pm} C(\Gamma)$ is a hyperbolic metric. 
In fact, $\partial^{\pm} C(\Gamma)$ are totally geodesic surfaces bent along geodesic laminations.

A closed subset in a hyperbolic surface $S$ is called a \emph{geodesic lamination} if it is a union of disjoint simple geodesics.
A \emph{measured lamination} is a geodesic lamination with a full-support transverse measure.
We denote the set of all compactly supported measured laminations by $\mathcal{ML}(S)$ equipped with the weak* topology.
A basic example is a simple closed geodesic with a transverse Dirac measure, which is regarded as a weighted simple closed curve on $S$.
It is known that the set of weighted simple closed curves is dense in $\mathcal{ML}(S)$.

For a quasi-Fuchsian group $\Gamma$, the \emph{bending lamination} $pl^{\pm}(\Gamma)$ is the measured lamination on $\partial^{\pm} C(\Gamma)$ 
such that the complement of the support of $pl^{\pm}(\Gamma)$ consists of totally geodesic surfaces 
and the transverse measure coincides with the exterior angle between these totally geodesic pieces.
The support of $pl^{\pm}(\Gamma)$ is called the \emph{bending locus}.
We remark that we can recover $\Gamma$ up to conjugation from the pair $\partial^{+} C$ and $pl^+(\Gamma)$ (resp. the pair $\partial^{-} C$ and $pl^-$)
by considering the developing map of the \emph{pleated surface} determined by this pair \cite{epstein-marden}.

\subsection{Deformation space}
Let $R_{par}(S)$ be the set of representations of $\pi_1(S)$ into $\PSLC$ taking peripheral elements to parabolic elements.
From a presentation of $\pi_1(S)$, $R_{par}(S)$ is defined to be an affine algebraic set.
The \emph{character variety} is the algebro-geometric quotient of $R_{par}(S)$ under the conjugation action of $\PSLC$.
If we restrict our attention to irreducible representations, the character variety is nothing but the usual quotient of $R_{par}(S)$ by the conjugation action. 
We denote this quotient of the set of irreducible representations by $X_{par}(S)$. 
For a representation $\rho$, we denote its conjugacy class by $[\rho]$.

A representation $\rho$ is called \emph{quasi-Fuchsian} (resp. \emph{Fuchsian}) if $\rho$ is faithful and $\rho(\pi_1(S))$ is a quasi-Fuchsian (resp. Fuchsian) group.
We denote the subset of $X_{par}(S)$ consisting of quasi-Fuchsian representations by $\mathcal{QF}(S)$.
It is known that $\mathcal{QF}(S)$ is contained in the set of discrete faithful representations $AH(S)$ and $\mathrm{Int}(AH(S)) = \mathcal{QF}(S)$.
Moreover the closure of $\mathcal{QF}(S)$ coincides with $AH(S)$ by the resolution of the Density Conjecture.
(For our main purpose, the once-punctured torus case, this is due to Minsky \cite{minsky_opt}.)

\subsection{Complex length and linear slice}
\label{subsec:complex_length}
Recall that $\PSLC$ acts on the hyperbolic space $\mathbb{H}^3$ as orientation preserving isometries.
For a loxodromic or hyperbolic element $A \in \PSLC$, 
we define the \emph{complex length} of $A$ by $\lambda(A) = l(A) + a(A) \sqrt{-1}$ where $l(A)$ is the translation distance of $A$ and $a(A)$ is the rotation angle.
So $\lambda(A)$ is defined modulo $2 \pi \sqrt{-1} \mathbb{Z}$.
It is easy to see that $A$ is conjugate to an upper triangular matrix $\begin{pmatrix} e^{\lambda(A)/2} & * \\ 0 & e^{-\lambda(A)/2} \end{pmatrix}$.
Using this presentation, we extend the definition of $\lambda(A)$ for elements other than loxodromic or hyperbolic ones.

A simple closed curve is said to be \emph{essential} if it is not homotopic to a point or a peripheral curve.
Let $\mathcal{S}(S)$ be the set of all isotopy classes of unoriented essential simple closed curves on $S$.
For $\gamma \in \mathcal{S}(S)$, we take an element $g \in \pi_1(S)$ homotopic to $\gamma$,
then define a function $\lambda_\gamma$ on $X_{par}(S)$ with values in $\mathbb{C} / (2 \pi \sqrt{-1} \mathbb{Z})$ by $\lambda_\gamma(\rho) = \lambda (\rho(g))$. 
We have
\begin{equation}
\pm \tr (\rho) = e^{\lambda_{\rho}/2} + e^{-\lambda_{\rho}/2} = 2 \cosh \left( \frac{\lambda_{\rho}}{2} \right).
\end{equation}
We remark that 
\[
\cosh \left( \frac{\lambda + 2 \pi \sqrt{-1}}{2} \right) = - \cosh \left( \frac{\lambda}{2} \right) ,
\]
which is consistent with the fact that the trace is only well-defined up to sign for a $\PSLC$-representation.

For $\gamma \in \mathcal{S}(S)$ and a complex number $\lambda$, we define the \emph{linear slice} $X_{\gamma}(\lambda) \subset X_{par}(S)$ by 
\[
X_{\gamma} (\lambda) = \{ [\rho] \in X_{par}(S) \mid \lambda_{\gamma}(\rho) = \lambda \}.
\]
(We remark that the linear slice is sometimes defined by $X_{\gamma}(\lambda) \cap AH(S)$ in the literature.)
If $\gamma$ is clear from the context, we omit the subscript.
For instance, if $S$ is a once-punctured torus, all essential simple closed curves are related by homeomorphisms of $S$, 
thus it is not important to indicate the curve $\gamma$.

We denote the intersection $\mathcal{QF}(S) \cap X_{\gamma}(\lambda)$ by $\mathcal{QF}_{\gamma}(\lambda)$ or $\mathcal{QF}(\lambda)$ 
if $\gamma$ is clear from the context. 
In this paper, we shall study the shape of $\mathcal{QF}(\lambda)$ in $X(\lambda)$.
It is shown in \cite[Proposition 4.3]{komori-yamashita} that
if $X(\lambda)$ is 1-dimensional (i.e. $S$ is a once-punctured torus or a four-times punctured sphere), each component of $\mathcal{QF}(\lambda)$ is an open disk.
This follows from a theorem of McMullen \cite[Theorem 5.1]{mcmullen}.
\begin{theorem}[McMullen]
\label{thm:mcmullen_disk_convexity}
The space of quasi-Fuchsian groups $\mathcal{QF}(S)$ is disk-convex in $X_{par}(S)$.
That is, every continuous map $f$ from the closed disk $\overline{\Delta}$ to $\mathcal{QF}(S)$ such that $f |_{\Delta}$ is holomorphic 
and $f(\partial \overline{\Delta}) \subset \mathcal{QF}(S)$ implies $f(\overline{\Delta}) \subset \mathcal{QF}(S)$.
\end{theorem}
From Theorem \ref{thm:mcmullen_disk_convexity}, each component of $\mathcal{QF}(\lambda)$ is a simply connected domain, thus it is a disk.

\section{Character variety of a once-punctured torus and trace functions}
\label{sec:character_variety}
In this section, we collect some known facts on the character variety of a once-punctured torus $S = S_{1,1}$. 
\subsection{Essential simple closed curves}
\label{subsec:simple_closed_curves}
We fix a system of generators $a, b$ of $\pi_1(S)$ so that the commutator $[a,b]$ is homotopic to the puncture.
For convenience, we give an orientation on $S$ so that the direction from $a$ to $b$ is anti-clockwise.
The homology classes $[a], [b]$ form a basis of $H_1(S;\mathbb{Z})$.
An essential simple closed curve gives a homology class $\pm(p [a] + q [b])$, 
and then a unique element $p/q \in \widehat{\mathbb{Q}} = \mathbb{Q} \cup \{ \infty \}$.
This gives an identification between the set $\mathcal{S}(S)$ of isotopy classes of unoriented essential simple closed curves and $\widehat{\mathbb{Q}}$.

For $p/q \in \widehat{\mathbb{Q}}$, we denote the right Dehn twist along $p/q$ by $D_{p/q}$.
The mapping class group of $S$ acts on $H_1(S; \mathbb{Z}) \cong \mathbb{Z}^2$ faithfully, and is regarded as $\mathrm{SL}(2, \mathbb{Z})$.
In particular, we have 
\[
D_{1/0} = \begin{pmatrix} 1 & 1 \\ 0 & 1 \end{pmatrix}, \quad
D_{0/1} = \begin{pmatrix} 1 & 0 \\ -1 & 1 \end{pmatrix}.
\]

\subsection{$\SLC$ and $\PSLC$ character variety}
The $\SLC$-character variety is similarly defined as in the case of $\PSLC$. 
Let $R^{SL}_{par}(S)$ be the set of representations of $\pi_1(S)$ into $\SLC$ taking peripheral elements to parabolic elements.
The $\SLC$-character variety is the algebro-geometric quotient of $R^{SL}_{par}(S)$ under the conjugation action of $\SLC$.
We denote the quotient of the set of irreducible representations of $R^{SL}_{par}(S)$ by $X^{SL}_{par}(X)$.

It is known that the map 
$[\rho] \mapsto (\tr(\rho(a)), \tr(\rho(b)), \tr(\rho(ab)))$ gives an isomorphism of varieties
\begin{equation}
\label{eq:SL_character_variety}
X^{SL}_{par}(S) \xrightarrow[]{\cong} \{ (x,y,z) \in \mathbb{C}^3 \mid x^2 + y^2+ z^2 - x y z = 0 \}.
\end{equation}
Thus we regard $X^{SL}_{par}(S)$ as the algebraic variety defined in the right hand side of (\ref{eq:SL_character_variety}).
The cohomology group $H^1(S; \mathbb{Z}/2 \mathbb{Z})$ acts on $X^{SL}_{par}(S)$ and the quotient by this action 
is known to be isomorphic to the $\PSLC$-character variety $X_{par}(S)$.
Explicitly, 
the generators $[a], [b]$ of $H^1(S; \mathbb{Z}/2 \mathbb{Z})$ act on $X^{SL}_{par}(S)$ as 
\begin{equation}
\label{eq:action_of_H^1(S)}
(x,y,z) \mapsto (-x,y,-z), \quad
(x,y,z) \mapsto (x,-y,-z),
\end{equation}
respectively.

\subsection{Complex Fenchel-Nielsen coordinates}
The complex Fenchel-Nielsen coordinates are defined for quasi-Fuchsian groups in \cite{kourouniotis} and \cite{tan}.
For a once-punctured torus case, Parker and Parkkonen \cite{parker-parkkonen} gave an explicit presentation of the complex Fenchel-Nielsen coordinates.
But we remark that the parameter $\lambda$ in \cite{parker-parkkonen} is half of ours.

Let $\rho$ be a quasi-Fuchsian representation of the once-punctured torus $S$.
Let $C_a$ be a simple closed curve representing the conjugacy class of $a$.
The \emph{length parameter} $\lambda$ is the complex length of $C_a$, 
and the \emph{twist parameter} $\tau$ is naturally defined as a complexification of the classical Fenchel-Nielsen twist parameter.
Anyway, this pair of coordinates gives a biholomorphic map from $\mathcal{QF}(S)$ to a domain in $\mathbb{C}^2$.
The inverse of this map has the following form
\begin{equation}
\label{eq:pre_complexFN}
(\lambda, \tau) \mapsto 
\left( 2 \cosh ( \lambda / 2 ), \, 
\frac{2 \cosh ( \tau / 2 )}{\tanh(\lambda/2)}, 
\, \frac{2 \cosh ( (\tau +\lambda)/ 2 )}{\tanh(\lambda/2)} \right),
\end{equation}
where we regard the target $\mathcal{QF}(S) \subset X^{PSL}_{par}(S)$ as a quotient of $X^{SL}_{par}(S)$ by the action of $H^1(S;\mathbb{Z}/2\mathbb{Z})$ as in (\ref{eq:action_of_H^1(S)}) 
and $X^{SL}_{par}(S)$ as a subset of $\mathbb{C}^3$ by (\ref{eq:SL_character_variety}).
The map defined by (\ref{eq:pre_complexFN}) can be naturally extended to a holomorphic map $(\mathbb{C} \setminus \{0\}) \times \mathbb{C} \to  X^{PSL}_{par}(S)$.
We take a domain in $(\mathbb{C} \setminus \{0\}) \times \mathbb{C}$ which injects into $X^{PSL}_{par}(S)$. 
\begin{proposition}
\label{prop:complexFN}
For $b_0 \in \mathbb{R}$, let 
\begin{equation}
\label{eq:domain_of_definition}
\mathcal{D} = \{ \lambda \in \mathbb{C} \mid  \mathrm{Re}(\lambda) > 0, \,  -\pi < \mathrm{Im}(\lambda) < \pi  \} 
\times \{ \tau \in \mathbb{C} \mid b_0 \leq \mathrm{Im}(\tau) < b_0 + 2 \pi\}.
\end{equation}
Then the map $\psi_{FN} : \mathcal{D} \to X^{PSL}_{par}(S)$ defined by
\begin{equation}
\label{eq:coplexFN_for_PSL}
\psi_{FN}( (\lambda, \tau) ) = \left( 2 \cosh ( \lambda / 2 ), \,
\frac{2 \cosh ( \tau / 2 )}{\tanh(\lambda/2)}, \, \frac{2 \cosh ( (\tau +\lambda)/ 2 )}{\tanh(\lambda/2)} \right),
\end{equation}
is injective. 
The image $\psi_{FN} (\mathcal{D})$ contains $\mathcal{QF}(S)$ entirely.
If we fix $\lambda$ such that $\mathrm{Re}(\lambda) > 0$ and $-\pi < \mathrm{Im}(\lambda) < \pi$, $\psi_{FN}$ takes values in $X(\lambda) = X_{1/0}(\lambda)$, 
and the map $\psi_{FN}(\lambda, \cdot ) : \{ \tau \in \mathbb{C} \mid b_0 \leq \mathrm{Im}(\tau) < b_0 + 2 \pi\} \to  X(\lambda)$ is a bijection.
\end{proposition}
\begin{remark}
We do not assume that $\psi_{FN}^{-1}(\mathcal{QF}(S)) \subset \mathcal{D}$ is connected.
We also remark that the description of $\psi_{FN}$ is very simple but it is difficult to determine $\psi_{FN}^{-1}(\mathcal{QF}) \subset \mathcal{D}$.
\end{remark}
\begin{proof}
We can prove directly from the form (\ref{eq:coplexFN_for_PSL}) of $\psi_{FN}$,
but here we use the explicit parametrization given in \cite[\S 8.2]{kabaya} since the techniques are more applicable for general surfaces.
Let $E = \mathbb{C} \setminus \{0, \, \pm 1, \, \pm \sqrt{-1}\}$.
Using the variable $e_1$ for $E$ and $t_1$ for $\mathbb{C}^*$, 
we define an algebraic map $E \times \mathbb{C}^* \to X^{PSL}_{par}(S)$ by
\begin{equation}
\label{eq:explicit_matrices}
\rho(a) = \begin{pmatrix} e_1 & 2 e_1^{-1} \\ 0 &e_1^{-1} \end{pmatrix}, \quad 
\rho(b) = \frac{1}{\sqrt{t_1}(e_1^2-1)} \begin{pmatrix} (e_1^2+1)t_1 + 2 & -2(t_1+1) \\ -e_1^2+1 & e_1^2-1 \\ \end{pmatrix}. 
\end{equation}
(We remark that the condition $e_1 \neq \pm \sqrt{-1}$ is equivalent to $- {e_1}^2 \neq 1$, which was described in (8) of \cite[\S 5.1]{kabaya}.)
This map becomes injective (see \cite[Theorem 2]{kabaya}) after taking the quotient by two $\mathbb{Z}/2\mathbb{Z}$-actions defined by 
$(e_1,t_1) \mapsto ({e_1}^{-1}, {t_1}^{-1})$ ((34) of \cite[\S 9.2]{kabaya}) and $e_1 \mapsto -e_1$ (\cite[\S 3.2]{kabaya}).
Thus it is injective on $\{ e_1 \in \mathbb{C} \mid |e_1|>1, \, \mathrm{Im}(e_1)>0  \} \times \mathbb{C}^*$.

Define a map $\mathcal{D} \to E \times \mathbb{C}^*$ by $(\lambda, \tau) \mapsto (e_1, t_1) = (\exp(\lambda/2), \exp(\tau))$, 
which maps $\mathcal{D}$ to $\{ e_1 \in \mathbb{C} \mid |e_1|>1, \, \mathrm{Im}(e_1)>0  \} \times \mathbb{C}^*$ injectively.
By the above result, the composition 
\[
\mathcal{D} \to \{ e_1 \in \mathbb{C} \mid |e_1|>1, \, \mathrm{Im}(e_1)>0  \} \times \mathbb{C}^* \subset 
E \times \mathbb{C}^* \to X^{PSL}_{par}(S)
\] 
is injective.
By direct calculation, we can check that this composition coincides with $\psi_{FN}$.

Since by \cite[Theorem 2]{kabaya} the image of $E \times \mathbb{C}^* \to X^{PSL}_{par}(S)$ contains all representations such that 
$\rho|_{\pi_1 (S \setminus C_a)}$ is irreducible and $\rho(a)$ is hyperbolic or loxodromic, 
$\psi_{FN} (\mathcal{D})$ contains $\mathcal{QF}(S)$ entirely.
If we fix $\lambda$, it is clear that $\psi_{FN}$ takes values in $X(\lambda) = X_{1/0}(\lambda)$.
For the surjectivity of $\psi_{FN}(\lambda, \cdot )$, 
it is sufficient to show that if $\rho \in X(\lambda)$, then $\rho|_{\pi_1 (S \setminus C_a)}$ is irreducible. 
We fix generators $a$ and $b^{-1} a b$ for $\pi_1 (S \setminus C_a)$.
Assume that $\rho|_{\pi_1 (S \setminus C_a)}$ is reducible, then we can conjugate $\rho$ such that 
\[
\rho(a) = \begin{pmatrix} e_1 & 0 \\ 0 & e_1^{-1} \end{pmatrix}, \quad 
\rho(b^{-1} a b ) = \begin{pmatrix} x_1 & x_2 \\ 0 & x_4 \end{pmatrix},
\]
where $e_1 = \exp(\lambda/2) \neq 0, \pm 1, \pm \sqrt{-1}$.
Since $\rho(a)^{-1} \rho(b^{-1} a b ) = \begin{pmatrix} e_1^{-1} x_1 & * \\ 0 & e_1 x_4 \end{pmatrix}$ is parabolic, $(x_1, x_4) = \pm (e_1, e_1^{-1})$.
If we let $\rho(b) = \begin{pmatrix} y_1 & y_2 \\ y_3 & y_4 \end{pmatrix}$,
from $\rho(b)^{-1} \rho(a) \rho(b) = \rho(b^{-1} a b)$, we have 
\[
\begin{pmatrix} e_1 y_1 y_4 - e_1^{-1} y_2 y_3 & (e_1 - e_1^{-1}) y_2 y_4 \\ (e_1^{-1} - e_1) y_1 y_3 & e_1^{-1} y_1 y_4 - e_1 y_2 y_3 \end{pmatrix}
= \begin{pmatrix} \pm e_1 & * \\ 0 & \pm e_1^{-1} \end{pmatrix},
\]
thus $y_1 = 0$ or $y_3 = 0$. If $y_1 = 0$, we have $\pm e_1 ^2 = - y_2 y_3 =  \pm e_1 ^{-2}$. 
This contradicts $e_1 \neq \pm 1, \pm \sqrt{-1}$.
We conclude that $y_3= 0$, thus $\rho$ is reducible.
\end{proof}

Next, we consider the action of the mapping class group $\mathrm{SL}(2,\mathbb{Z})$.
By (37) of \cite{kabaya}, the Dehn twist $D_{1/0}$ acts on the parameters $(e_1,t_1)$ in the proof of Proposition \ref{prop:complexFN} as 
$D_{1/0} (e_1, t_1) = (e_1, e_1^2 t_1)$.
Thus on $\mathcal{D}$, it acts as 
\begin{equation}
\label{eq:dehn_twist_action}
D_{1/0} (\lambda, \tau) = (\lambda, \tau + \lambda),
\end{equation}
modulo $2 \pi \sqrt{-1} \mathbb{Z}$ on the second factor.
The action by the matrix $\begin{pmatrix} 0 & -1 \\ 1 & 0 \end{pmatrix}$ is much more complicated (see \cite[Proposition 10]{kabaya}).

\subsection{Trace functions}
\label{subsec:trace_functions}
In this subsection, we show some properties of trace functions on the character variety.
Since an element of $\PSLC$ determines a matrix only up to sign, we need to fix a lift to the $\SLC$-character variety.

The map defined by (\ref{eq:coplexFN_for_PSL}) naturally lifts to a map $\widetilde{\psi}_{FN} : \mathcal{D} \to X^{SL}_{par}(S)$ 
by regarding the target as the $\SLC$-character variety via the isomorphism (\ref{eq:SL_character_variety}).
We remark that the lift is natural with respect to the expression (\ref{eq:coplexFN_for_PSL}), 
but there is no canonical way to lift a subset of $X^{PSL}_{par}(S)$ to $X^{SL}_{par}(S)$.
The map $\widetilde{\psi}_{FN}$ can be extended to a holomorphic map $(\mathbb{C} \setminus \{0,1\}) \times (\mathbb{C} \setminus \{0\}) \to  X^{SL}_{par}(S)$.
For $\tau' = \tau + 2 \pi \sqrt{-1}$, we have
\[
\begin{split}
&\left( 2 \cosh ( \lambda / 2 ), \, 
\frac{2 \cosh ( \tau' / 2 )}{\tanh(\lambda/2)}, 
\, \frac{2 \cosh ( (\tau' +\lambda)/ 2 )}{\tanh(\lambda/2)} \right) \\
&=
\left( 2 \cosh ( \lambda / 2 ), \, 
-\frac{2 \cosh ( \tau / 2 )}{\tanh(\lambda/2)}, 
\, - \frac{2 \cosh ( (\tau +\lambda)/ 2 )}{\tanh(\lambda/2)} \right). 
\end{split}
\]
This means that $(\lambda, \tau)$ and $(\lambda, \tau + 2 \pi \sqrt{-1})$ are the same in $X^{PSL}_{par}(S)$ but not in $X^{SL}_{par}(S)$.

Fix an essential simple closed curve $p/q \in \widehat{\mathbb{Q}}$.
We take an element $g_{p/q} \in \pi_1(S)$ in the conjugacy class representing the curve $p/q$ (we do not care about the orientation). 
For $[\rho] \in X^{SL}_{par}(S)$, we define $\tr_{p/q}( [\rho] ) = \tr(\rho( g_{p/q}))$.
On $\mathcal{D}$, we define the trace function by $\tr_{p/q} (\widetilde{\psi}_{FN}(\lambda, \tau))$.
In particular, we have
\[
\tr_{1/0}(\lambda, \tau) = 2 \cosh \left( \frac{\lambda}{2} \right), \quad
\tr_{0/1}(\lambda, \tau) = \frac{2 \cosh ( \tau/2 ) }{\tanh(\lambda/2)}, \quad
\tr_{1/1}(\lambda, \tau) = \frac{2 \cosh ( (\tau+\lambda)/2 ) }{\tanh(\lambda/2)}
\]
from (\ref{eq:coplexFN_for_PSL}).
Considering the action of the Dehn twist (\ref{eq:dehn_twist_action}), we have
\begin{equation}
\label{eq:traces_of_integral_curves}
\tr_{n/1}(\lambda, \tau) = \frac{2 \cosh ( (\tau+n\lambda)/2 ) }{\tanh(\lambda/2)}.
\end{equation}
This can be also computed from trace identities.

\begin{proposition}
\label{prop:trace_fuction}
The trace function $\tr_{p/q}$ has the following form.
\begin{equation}
\label{eq:trace_function}
\tr_{p/q} (\lambda, \tau) = \sum_{k=-q}^q c_{p/q,k}(\lambda) \exp \left( \frac{k \tau}{2} \right) 
\end{equation}
where $c_{p/q,k}(\lambda)$ $(k=0, \pm 1, \cdots, \pm q)$ are some holomorphic functions. 
For $l \in \mathbb{R}_{>0}$,  $c_{p/q,k}(l)$ are real and $c_{p/q, \pm q}(l) \neq 0$.
\end{proposition}
\begin{proof}
This is proved by induction.
It is clear if $q=0$ or $1$ since $\tr_{1/0} = 2 \cosh(\lambda/2)$ and 
\[
\tr_{n/1}(\lambda, \tau) = \frac{2 \cosh(\frac{\tau + n \lambda}{2})}{\tanh(\lambda/2)} 
= \frac{\exp(n \lambda/2)}{\tanh(\lambda/2)} \exp \left( \frac{\tau}{2} \right) + \frac{\exp(-n\lambda/2)}{\tanh(\lambda/2)} \exp \left( \frac{-\tau}{2} \right).
\]

Following \cite[\S 3.1]{keen-series93}, any simple closed curve $p/q$ can be represented by a \emph{special word} in the generators $a, b \in \pi_1(S)$.
In the remaining of this proof, we always express a rational number $p/q$ by a unique pair of coprime integers $p$, $q$ with $q > 0$.
A pair of rational numbers $(p/q$, $r/s)$ is called \emph{Farey neighbors} if $ps -qr = \pm 1$.
If $(p/q, r/s)$ are Farey neighbors, then both pairs $(p/q, (p+r)/(q+s))$ and $((p+r)/(q+s), r/s)$ are again Farey neighbors.
Moreover if $p/q < r/s$, then $p/q < (p+r)/(q+s) < r/s$.
All rational numbers are obtained in a unique way by repeated application of the process $(p/q, r/s) \mapsto (p/q, (p+r)/(q+s))$ 
starting with integer neighbors $(n/1, (n+1)/1)$.
For an integer $n/1$, we define the special word $g_{n/1} = a^{-n} b$.
If $(p/q, r/s)$ are Farey neighbors with $p/q < r/s$, then define $g_{(p+r)/(q+s)} = g_{r/s} g_{p/q}$.

Since $\tr(AB) + \tr(AB^{-1}) = \tr(A) \tr(B)$ for any $A, B \in \SLC$, we have
\begin{equation}
\label{eq:trace_identity}
\tr_{\frac{p+r}{q+s}} = \tr_{p/q} \tr_{r/s} - \tr_{\frac{p-r}{q-s}}.
\end{equation}
for any Farey neighbors $(p/q, r/s)$ with $p/q < r/s$ and $q, s>0$.
Here we have $c_{(p+r)/(q+s), \pm(q+s)}(\lambda) =c_{p/q, \pm q}(\lambda) c_{r/s, \pm s}(\lambda)$.
Since all trace functions corresponding to rational numbers are obtained starting with integer neighbors $(n/1, (n+1)/1)$, this completes the proof.
\end{proof}

The next corollary means that if we fix $l>0$, there exists a real locus of $\lambda_{p/q}(l, \tau)$ 
between $\frac{(2j - 1)\pi}{q} < \mathrm{Im}(\tau) < \frac{(2j + 1)\pi}{q}$ $(j \in \mathbb{Z}, \, |j| < q/2)$ if $|\mathrm{Re}(\tau)|$ is large.
\begin{corollary}
\label{cor:existence_of_real_locus_at_large_real_part}
Fix a real number $l>0$ and a rational number $p/q$. 
There exists $T > 0$ satisfying the following condition: 
If $|t| > T$ and $j  \in \mathbb{Z}$ $(|j| <  q/2)$, 
there exists $b = b(t, j) \in (\frac{(2j-1)\pi}{q}, \frac{(2j+1)\pi}{q})$ such that
\[
\mathrm{Im} ( \lambda_{p/q}(l, t + b \sqrt{-1}) ) \equiv 0 \mod 2 \pi \mathbb{Z}.
\]
\end{corollary}
\begin{proof}
Let $\tau = t + b \sqrt{-1}$. From (\ref{eq:trace_function}), we have
\[
\mathrm{Im} ( \tr_{p/q} (l, \tau) ) 
= \sum_{k=-q}^q c_{p/q,k}(l) \exp \left( \frac{k t}{2} \right) \sin \left( \frac{k b}{2} \right).
\]
Thus if $|t|$ is sufficiently large, the signs of $\mathrm{Im}(\tr_{p/q}(l, t + \frac{(2j \pm 1)\pi}{q} \sqrt{-1}))$ are different.
Therefore there exists $b = b(t,j) \in (\frac{(2j-1)\pi}{q}, \frac{(2j+1)\pi}{q})$ such that $\mathrm{Im}(\tr_{p/q}(l, t + b \sqrt{-1})) = 0$.
If we let $\lambda_{p/q} = l_{p/q} + \sqrt{-1} \, a_{p/q}$, we have
\[
\tr_{p/q}/2 = \cosh(\lambda_{p/q}/2) = \cosh(l_{p/q}/2) \cos(a_{p/q}/2) + \sqrt{-1} \sinh(l_{p/q}/2) \sin(a_{p/q}/2).
\]
Thus $\mathrm{Im}(\tr_{p/q}) = 0$ if and only if $\mathrm{Im}(\lambda_{p/q}) = a_{p/q} \equiv 0 \mod 2 \pi \mathbb{Z}$.
\end{proof}

\section{Linear slices}
\label{sec:linear_slices}
In this section, we investigate some properties of the linear slice $X(l)$ for a real number $l > 0$,
then we give our first proof of Theorem \ref{komori_yamashita_theorem}.

\subsection{Linear slices of real length}
\label{subsec:linear_slices_of_real_length}
In the following of this paper, we regard $\{ (l,\tau) \mid -\pi \leq \mathrm{Im}(\tau) < \pi \}$ as $X_{1/0}(l)$ 
via the map $\psi_{FN}$ of Proposition \ref{prop:complexFN}.
\begin{proposition}
\label{prop:not_on_pm_pi}
For any $l>0$,  $(l, \tau)$ is not quasi-Fuchsian if $|\mathrm{Im}(\tau)| = \pi$. In particular,
\[
\mathcal{QF}(l) \subset \{l\} \times \{ \tau \mid - \pi < \mathrm{Im}(\tau) < \pi \}.
\]
\end{proposition}
\begin{proof}
Let $\tau = t \pm \pi \sqrt{-1}$. 
Recall the explicit representation given by (\ref{eq:explicit_matrices}). 
In this situation, we have $e_1 = \exp(l/2)$ and $t_1 = \exp(\tau) = -\exp(t)$, 
thus $\rho(a)$ and $\rho(b)$ in (\ref{eq:explicit_matrices}) preserve $\mathbb{R}P^1 = \mathbb{R} \cup \{\infty\}$ as M\"{o}bius transformations.
Thus if we assume that $\rho$ is quasi-Fuchsian, it must be Fuchsian.
On the other hand, since $\tr_{0/1} (l, \tau) = \frac{2 \cosh(t/2 \pm (\pi \sqrt{-1})/2 )}{\tanh(l/2)} = \frac{\pm 2 \sinh(t/2)}{\tanh(l/2)}\sqrt{-1}$,  
the holonomy along $0/1$ is not purely hyperbolic.
This contradicts the fact that $\rho$ is Fuchsian.
\end{proof}

Since the real line corresponds to all Fuchsian representations satisfying $\lambda_{1/0} \equiv l$, $\mathcal{QF}(l)$ contains $\mathbb{R}$.
Moreover, Parker and Parkkonen proved the following theorem by giving an explicit fundamental domain using Maskit's combination.
\begin{theorem}[Theorem 4.1 of \cite{parker-parkkonen}]
\label{thm:parker-parkkonen}
If $(l, \tau)$ satisfies
\[
|\mathrm{Im}(\tau)| < 2 \arccos( \tanh(l/2) ), 
\]
then $(l, \tau)$ is quasi-Fuchsian.
The representation corresponds to $(l, n l + 2 \arccos( \tanh(l/2) ) \sqrt{-1})$ is a cusp group for $-n/1$-curve.
\end{theorem}
At least, it is easy to check that the trace of $-n/1$-curve at $(l, n l + 2 \arccos( \tanh(l/2) ) \sqrt{-1})$ is $2$ from (\ref{eq:traces_of_integral_curves}).
\begin{proposition}
\label{prop:elliptic_loci}
If $b \in \mathbb{R}$ satisfies
\begin{equation}
\label{eq:integral_curves_to_be_elliptic}
2 \arccos(\tanh(l/2)) < b < \pi,
\end{equation}
then $(l, nl + b\sqrt{-1})$ does not corresponds to a quasi-Fuchsian representation.
If
\begin{equation}
\label{eq:half_integral_curves_to_be_elliptic}
\arccos \left( (\tanh(l/2))^2 - \frac{1}{\cosh(l/2)} \right) < b < \pi,
\end{equation}
then $(l, (n+1/2)l + b\sqrt{-1})$ does not corresponds to a quasi-Fuchsian representation.
\end{proposition}
We do not use the second assertion in later arguments. 
From Theorem \ref{thm:parker-parkkonen}, for any $l>0$ we have 
\[
2 \arccos(\tanh(l/2)) <\arccos \left( (\tanh(l/2))^2 - \frac{1}{\cosh(l/2)} \right).
\]
\begin{proof}
In the first case, since
\[
\tr_{-n/1}(l, nl + b\sqrt{-1}) = \frac{2\cosh(\frac{b\sqrt{-1}}{2})}{\tanh(l/2)} = \frac{2\cos(b/2)}{\tanh(l/2)}, 
\]
we have $0 < \tr_{-n/1}(l, nl + b\sqrt{-1}) < 2$ if (\ref{eq:integral_curves_to_be_elliptic}) is satisfied.
This means that the holonomy along $-n/1$-curve is elliptic, thus the representation is not in $\mathcal{QF}$.

In the second case, the trace function of $(-n-1/2)$-curve is
\[
\tr_{-n-1/2} = \tr_{-(n+1)/1} \tr_{-n/1} - \tr_{1/0} 
= \frac{2 \cosh(\frac{\tau - (n+1)\lambda}{2})}{\tanh(\frac{\lambda}{2})} \frac{2 \cosh(\frac{\tau - n \lambda}{2})}{\tanh(\frac{\lambda}{2})} - 2 \cosh(\lambda/2),
\]
thus we have
\[
\begin{split}
\tr_{-n-1/2}(l, (n+1/2)l + b\sqrt{-1})
=& \frac{2 \cosh(\frac{-l/2+\sqrt{-1}b}{2})}{\tanh(l/2)} \frac{2 \cosh(\frac{l/2+\sqrt{-1}b}{2})}{\tanh(l/2)} - 2 \cosh(l/2) \\
=& \frac{2}{(\tanh(l/2))^2} (\cosh(l/2) + \cos(b)) - 2 \cosh(l/2).
\end{split}
\]
Since $( \tanh(l/2) )^2 (\cosh(l/2)-1) < \cosh(l/2)-1$, we have  
\[
-2 < \frac{2}{( \tanh(l/2) )^2} (\cosh(l/2)-1) - 2 \cosh(l/2).
\]
Thus if 
\[
\frac{2}{( \tanh(l/2) )^2} (\cosh(l/2)-1) - 2 \cosh(l/2) < \tr_{-n-1/2} (l, (n+1/2)l + b\sqrt{-1}) < 2,
\]
the holonomy along $(-n-1/2)$-curve is elliptic, thus the corresponding representation is not quasi-Fuchsian. This is equivalent to 
\[
-1 < \cos(b) < (\tanh(l/2))^2(1 + \cosh(l/2)) - \cosh(l/2).
\]
Since the right hand side is equal to $(\tanh(l/2))^2  - \frac{1}{\cosh(l/2)}$, this is equivalent to (\ref{eq:half_integral_curves_to_be_elliptic}).
\end{proof}

\begin{corollary}
\label{cor:infinitely_many_by_dehn_twists}
If $\mathcal{QF}(l)$ has a non-standard component, then there are infinitely many.
\end{corollary}
\begin{proof}
By Theorem \ref{thm:parker-parkkonen} and Proposition \ref{prop:elliptic_loci}, 
a non-standard component is contained in a region $\{ \tau \mid n l < \mathrm{Re} (\tau) < (n+1) l \}$ for some integer $n$.
Since the Dehn twist along the curve $1/0$ acts on $X(l)$ as $\tau \mapsto \tau + l $ (see (\ref{eq:dehn_twist_action})), 
if there exists a non-standard component, there are infinitely many.
\end{proof}

\subsection{BM slices and pleating rays}
The component of $\mathcal{QF}(l)$ containing Fuchsian representations (equivalently, containing the real line) is relatively well-understood by the theory of 
pleating rays developed by Keen and Series \cite{keen-series04}.
We call this component the \emph{BM-slice} and denote it by $BM$.
We also call $BM$ the \emph{standard component} following the paper \cite{komori-yamashita}.
Theorem \ref{thm:parker-parkkonen} implies:
\begin{corollary}[Parker-Parkkonen]
\label{cor:BM_contains_parker_parkkonen_region}
$BM$ contains $\{ \tau \mid  |\mathrm{Im}(\tau)| < 2 \arccos(\tanh(l/2)) \}$.
\end{corollary}
Thus if $l$ is small, $BM$ is close to $\{ \tau \mid -\pi < \mathrm{Im}(\tau) < \pi \}$ 
and getting thinner and thinner as $l \to \infty$.

We define two subsets of $\mathcal{QF}(l)$ by
\[
BM^{\pm} = \{ \rho \in \mathcal{QF} \mid [pl^{\pm}(\rho)] = 1/0, \quad \lambda_{1/0}(\rho) = l \},
\]
where $[pl^{\pm}(\rho)] = 1/0$ means that the projective classes of $pl^{\pm}(\rho)$ and $1/0$ are equivalent.
As the notation suggests, the following holds (Theorem 5 and \S 10 of \cite{keen-series04}).
\begin{theorem}[Keen-Series]
\label{thm:BM_is_foliated_by_pleating_rays}
The real line $\mathbb{R}$ divides $BM$ into two open disks $BM^{\pm}$.
\end{theorem}
\begin{remark}
The terminology BM-slice is used for each of $BM^{\pm}$ in \cite{keen-series04}.
\end{remark}
\begin{corollary}
\label{cor:pleating_loci}
Suppose $\rho \in \mathcal{QF}(l)$ and $\rho$ is not Fuchsian. 
Then $\rho$ is in the BM-slice if and only if one of $[pl^{\pm}(\rho)]$ is $1/0$.
\end{corollary}
These are further decomposed as
\[
BM^{\pm} =  \bigsqcup_{[\mu] \in \mathcal{PML} \setminus \{1/0\}} \mathcal{P}^{\pm}_{\mu}, \quad 
\]
where we define the \emph{pleating rays} $\mathcal{P}^{\pm}_{\mu}$ for $[\mu] \in \mathcal{PML}(S) \setminus \{1/0\}$ by
\[
\mathcal{P}^{\pm}_{\mu} = \{ \rho \in \mathcal{QF} \mid [pl^{\pm}(\rho)] = 1/0, \quad \lambda_{1/0}(\rho) = l, \quad [pl^{\mp}(\rho)] =  [\mu] \}.
\]
If $[\mu]$ is a simple closed curve corresponding to $p/q$, we denote it as $\mathcal{P}^{\pm}_{p/q}$, and call a \emph{rational pleating ray}. 
The following was shown in Theorems 6 and 4 of \cite{keen-series04}.
\begin{theorem}[Keen-Series]
\label{thm:rational_pleating_rays_are_dense}
Each pleating ray is an open interval. 
The rational pleating rays $\mathcal{P}^{\pm}_{p/q}$ are dense in $BM^{\pm}$.
\end{theorem}
We also refer to \cite{komori-parkkonen} for further properties of BM-slices.
By definition, a rational pleating ray lies on a subset of the real locus of a complex length function, which is essentially given by a polynomial function 
(see the proof of Proposition \ref{prop:trace_fuction}).
This fact enables us to draw the shape of pleating rays $\mathcal{P}^{\pm}_{p/q}$, and thus $BM^{\pm}$ by explicit computation.

\subsection{First proof of Theorem \ref{komori_yamashita_theorem}}
\begin{proof}[Proof of Theorem \ref{komori_yamashita_theorem}]
As in \S \ref{subsec:linear_slices_of_real_length}, we regard $\{ (l,\tau) \mid -\pi \leq \mathrm{Im}(\tau) < \pi \}$ as $X(l)$.
Consider the length function $\lambda_{p/q}$ for $q \geq 4$.
We fix $l_0 > 0$ small enough so that $2 \arccos( \tanh(l_0/2) ) > 3 \pi / q$.
By Corollary \ref{cor:existence_of_real_locus_at_large_real_part}, for sufficiently large $t$, 
there exists $\tau = t + b \sqrt{-1}$ with $\pi /q < b < 3 \pi / q$ such that $\mathrm{Im} ( \lambda_{p/q}(l_0, \tau) ) = 0$. 
Thus if we let $l = \lambda_{p/q}(l_0, \tau) \in \mathbb{R}$, then $(l_0, \tau) \in X_{p/q}(l)$.

Since $|\mathrm{Im}(\tau)| = b < 3 \pi / q < 2 \arccos( \tanh(l_0/2) )$, 
$(l_0, \tau)$ is in the BM-slice by Corollary \ref{cor:BM_contains_parker_parkkonen_region}.
Therefore the supports of the bending laminations of $(l_0, \tau)$ are $1/0$ and some $p'/ q'$ by Corollary \ref{cor:pleating_loci}.
By taking $t$ sufficiently large, we assume that $(l_0, t + b\sqrt{-1})$ does not belong to $\mathcal{P}_{p/q}^{\pm}$, in particular $p'/q' \neq p/q$.
(We remark that a rational pleating ray $\mathcal{P}^{\pm}_{p/q}$, which is known to be 
an open interval by Theorem \ref{thm:rational_pleating_rays_are_dense}, is compactified by adding a Fuchsian group and a cusp group for each end.)

By applying a homeomorphism of $S$ which sends $p/q$ to $1/0$, the corresponding representation is in $\mathcal{QF}_{1/0}(l)$, 
but neither of the supports of the bending laminations is not $1/0$.
By Corollary \ref{cor:pleating_loci}, this is not in the BM-slice.
\end{proof}
\begin{remark}
In the proof, we can also consider $\lambda_{p/q}$ with $q = 3$ instead, but we assumed $q \geq 4$ to make the argument simple.
But if we assume $q = 1$ or $2$, the above argument does not work. 
\end{remark}

\section{Complex projective structures and complex earthquakes}
\label{sec:projective_structure}
In this section, we review the theory of complex projective structures and complex earthquakes. 
We refer to \cite{dumas}, \cite{ito_LMS} and \cite{mcmullen} for details.
We will give the second proof of Theorem \ref{komori_yamashita_theorem} in \S \ref{subsec:second_proof}.
\subsection{Complex projective structures}
A \emph{complex projective structure} (or $\mathbb{C}P^1$-structure) on a surface $S$ is a geometric structure defined by atlas of charts in $\mathbb{C}P^1$ 
whose transition functions are restrictions of $\PSLC$.
Since the action of $\PSLC$ on $\mathbb{C}P^1$ is holomorphic, a complex projective structure also gives a holomorphic structure on $S$.
We always assume that its underlying holomorphic structure around a cusp of $S$ is biholomorphic to some neighborhood of a punctured disk.
As in the definition of Teichm\"{u}ller space, we can define the space $P(S)$ of marked projective structures.
Namely, a marked projective structure is a $\mathbb{C}P^1$ structure $C$ with a marking homeomorphism $f:S \to C$, 
and two marked projective structures $(f_1, C_1)$ and $(f_2, C_2)$
are equivalent if there exists a map $h: C_1 \to C_2$ whose restriction to any chart is a M\"{o}bius transformation such that $f_1 \circ h$ is homotopic to $f_2$.
We denote the set of marked projective structures on $S$ by $P(S)$.
We denote the Teichm\"{u}ller space of $S$ by $\mathcal{T}(S)$ and the natural forgetful map by $\pi : P(S) \to \mathcal{T}(S)$.

The holonomy of a complex projective structure gives a $\PSLC$-representation of $\pi_1(S)$, which is unique up to conjugation.
By assumption, the holonomy around a puncture is parabolic.
Thus we have a map $\mathrm{hol} : P(S) \to X_{par}(S)$.

\subsection{Grafting and Thurston coordinates}
Let $\gamma$ be a simple closed curve on $S$ and $r$ a positive real number.
For $X \in \mathcal{T}(S)$, we define a marked complex projective structure $\mathrm{Gr}_{r \cdot \gamma}(X) \in P(S)$ as follows.
First, we replace $\gamma$ with a geodesic representative in its homotopy class with respect to the hyperbolic metric $X$.
$\mathrm{Gr}_{r \cdot \gamma}(X)$ is constructed by cutting $X$ along $\gamma$ and inserting a flat annulus of height $r$ with circumference $l_\gamma(X)$ without twisting.
This operation is called \emph{grafting}.

As we have seen in \S \ref{subsec:quasi-fuchsian}, the weighted simple closed curve $r \cdot \gamma$ is regarded as a measured lamination,
and the set of weighted simple closed curves is dense in $\mathcal{ML}(S)$.
The definition of $\mathrm{Gr}$ can be extended continuously to $\mathcal{ML}(S)$. 
For $\mu \in \mathcal{ML}(S)$ and $X \in \mathcal{T}(X)$, we denote the grafted surface by $\mathrm{Gr}_{\mu}(X)$.
The map $\mathrm{Gr} : \mathcal{ML}(S) \times \mathcal{T}(S) \to P(S)$ is continuous.
Moreover $\mathrm{Gr}$ is a homeomorphism, thus gives a global coordinate system for $P(S)$, called \emph{Thurston coordinates}.
\begin{theorem}[Thurston, Kamishima-Tan \cite{kamishima-tan}]
\label{thm:thurston_coordinates}
The grafting map
\[
\mathrm{Gr} : \mathcal{ML}(S) \times \mathcal{T}(S) \to P(S)
\]
is a homeomorphism.
\end{theorem}
The composition map $\mathcal{ML}(S) \times \mathcal{T}(S) \xrightarrow{\mathrm{Gr}}{} P(S) \xrightarrow{\pi}{} \mathcal{T}(S)$ is also referred to as a grafting map 
and usually denoted by $\mathrm{gr}$, but we do not use it in this paper.

\subsection{Complex projective structures with quasi-Fuchsian holonomy}
\label{subsec:cpx_proj_str_with_qf_hol}
By definition, $\mathrm{hol}^{-1}(\mathcal{QF}(S)) \subset P(S)$ is the set of marked complex structures with quasi-Fuchsian holonomy.
A complex projective structure with quasi-Fuchsian holonomy is called \emph{standard} if its developing map is injective, otherwise \emph{exotic}.
Let $Q_0 \subset \mathrm{hol}^{-1}(\mathcal{QF}(S))$ be the set of standard complex projective structures.
The holonomy map $\mathrm{hol} : Q_0 \to \mathcal{QF}(S)$ gives a diffeomorphism.

Let $\mathcal{ML}_{\mathbb{Z}}(S)$ be the set of isotopy classes of disjoint collections of essential simple closed curves with non-negative integral weights.
We can regard $\mathcal{ML}_{\mathbb{Z}}(S)$ as integral points of $\mathcal{ML}(S)$ in coordinates by train tracks.

For $X \in \mathcal{T}(S)$ and $\mu \in \mathcal{ML}_{\mathbb{Z}}(S)$,
the holonomy of $\mathrm{Gr}_{2 \pi \cdot \mu}(X)$ coincides with the one of $X$, in particular, has a Fuchsian holonomy.
More generally, for $Z \in Q_0$, since it is obtained by a quasi-conformal conjugation from a Fuchsian uniformization, 
we can always take an \emph{admissible} multicurve homotopic to $\mu \in \mathcal{ML}_{\mathbb{Z}}(S)$ and perform $2 \pi$-grafting along it 
(see \cite[\S 6]{GKM} for $2 \pi$-grafting along admissible curves.).
This operation does not change the holonomy, in particular, the resulting projective surface has a quasi-Fuchsian holonomy.
It only depends on the homotopy class of the admissible multicurve, and gives a diffeomorphism from $Q_0$ onto its image $Q_\mu \subset \mathrm{hol}^{-1}(\mathcal{QF}(S))$.
Goldman  \cite{goldman} showed that this construction gives a complete classification of complex projective structures with quasi-Fuchsian holonomy:
\begin{theorem}[Goldman]
\label{thm:goldman}
\[
\mathrm{hol}^{-1}(\mathcal{QF}(S)) = \bigsqcup_{\mu \in \mathcal{ML}_{\mathbb{Z}}(S)} Q_{\mu}.
\]
\end{theorem}
We call $Q_0$ the \emph{standard component}, and $Q_\mu \, (\mu \neq 0)$ the \emph{exotic component}.
We remark that If $(f, C) \in Q_{\mu}$ and $\phi$ is a diffeomorphism of $S$, 
then $(f \circ \phi, C)$ is in $Q_{\phi^{-1}(\mu)}$ since $f \circ \phi ( \phi^{-1}(\mu)) = f(\mu)$.
\begin{remark}
For non-zero $\mu, \mu' \in \mathcal{ML}_\mathbb{Z}(S)$, we can also perform $2 \pi$-grafting of an element of $Q_{\mu}$ along an admissible multicurve homotopic to $\mu'$.
In this case, the result may depend on the choice of the admissible curve in the homotopy class \cite{ito}, \cite{calsamiglia-deroin-francaviglia}.
\end{remark}

\subsection{Complex earthquakes}
First we recall the Fenchel-Nielsen twist deformation of the Teichm\"{u}ller space $\mathcal{T}(S)$.
Let $\gamma$ be a simple closed curve on $S$ and $t$ a real (possibly negative) number. 
For $X \in \mathcal{T}(S)$, we replace $\gamma$ with a geodesic in its homotopy class as in the case of grafting.
The Fenchel-Nielsen twisting is obtained by cutting the surface along $\gamma$ and gluing it back with twisting hyperbolic length $t$ to the right.
We denote the resulting marked hyperbolic surface by $\mathrm{tw}_{t \cdot \gamma}(X) \in \mathcal{T}(S)$.
This construction can be generalized for measured laminations, called \emph{earthquakes} \cite{kerckhoff}.

We denote the upper half plane by $\mathbb{H} = \{ \tau \in \mathbb{C} \mid \mathrm{Im}(\tau) > 0\}$ and its closure by 
$\overline{\mathbb{H}} = \{ \tau \in \mathbb{C} \mid \mathrm{Im}(\tau) \geq 0\}$.
For a measured lamination $\mu$ on $S$, the \emph{complex earthquake} is a map $\overline{\mathbb{H}} \times \mathcal{T}(S) \to P(S)$ defined by
\begin{equation}
\label{eq:complex_earthquake}
\mathrm{Eq}_{(t+b\sqrt{-1}) \cdot \gamma}(X) = \mathrm{Gr}_{b \cdot \gamma}( \mathrm{tw}_{t \cdot \gamma}(X)).
\end{equation}
(Here $t$ and $b$ imply `twisting' and `bending' respectively.)
\begin{remark}
If we let
\[
\ML_{\overline{\mathbb{H}}}(S) = \{ (\lambda, z) \in \ML(S) \times \overline{\mathbb{H}} \} / \langle (\lambda, t z) \sim  ( t \lambda, z) \mid  t \in \mathbb{R}_{>0} \rangle,
\]
$\mathrm{Eq}$ gives a continuous map $\mathcal{ML}_{\overline{\mathbb{H}}} \times \mathcal{T}(S) \to P(S)$.
The definition of $\mathcal{ML}_{\overline{\mathbb{H}}}(S)$ naturally extended to $\mathbb{C}$.
McMullen enlarged the domain of $\mathrm{Eq}$ to include some region in the lower half space \cite{mcmullen}.
Roughly, this is obtained by turning the two convex core boundaries upside-down.
\end{remark}

From now on, we focus on the case where $S$ is a once-punctured torus. 
We represent a simple closed curve on $S$ by a rational number as in \S \ref{subsec:simple_closed_curves}.
We assume that $\mu = 1/0$.
In this case, the complex earthquake can be interpreted as a lift of the complex Fenchel-Nielsen coordinates $\psi_{FN}$ as follows.
For a real number $l>0$, we take a marked hyperbolic surface $X_l$ so that $l_{1/0}(X_l) = l$, 
and the geodesic representatives of $1/0$ and $0/1$ are orthogonal.
As a point on $\mathcal{T}(S)$, $X_l$ is uniquely determined.
If we let $\tau = t + \sqrt{-1} b \in \overline{\mathbb{H}}$, the representation $\mathrm{hol}(\mathrm{Eq}_{\tau \cdot 1/0} (X_l))$ is obtained from 
$X_l$ by twisting distance $t$ and bending angle $b$ along $1/0$.
Thus we have 
\[
\mathrm{hol}(\mathrm{Eq}_{\tau \cdot 1/0} (X_l)) = \psi_{FN}(l, \tau)
\]
where $\psi_{FN}$ is given by the complex Fenchel-Nielsen coordinates (\ref{eq:coplexFN_for_PSL}).
In particular, $\mathrm{hol} ( \mathrm{Eq}_{\overline{\mathbb{H}} \cdot 1/0} (X_l) ) = X(l)$.
We summarize the discussion in the following:
\begin{proposition}
\label{prop:summary_of_complex_earthquake}
We regard $\overline{\mathbb{H}}$ as a subset of $P(S)$ via the complex earthquake (\ref{eq:complex_earthquake}) 
and $X(l)$ as $\{ \tau \mid -\pi \leq \mathrm{Im}(\tau) < \pi \}$ by Proposition \ref{prop:complexFN}.
Then the restriction of the holonomy map 
$\mathrm{hol} : \overline{\mathbb{H}}   \to    \{ \tau \mid -\pi \leq \mathrm{Im}(\tau) < \pi \}$ is regarded as $\tau \mapsto \tau \mod 2 \pi \sqrt{-1} \mathbb{Z}$.
\begin{equation}
\label{eq:complex_earthquake_to_linear_slice}
\begin{matrix}
P(S) & \xrightarrow[]{\mathrm{hol}} & X(S) \\
\rotatebox{90}{$\subset$} & &  \rotatebox{90}{$\subset$} \\
\mathrm{Eq}_{\overline{\mathbb{H}} \cdot 1/0} (X_l) \cong \overline{\mathbb{H}} & \to  &  X(l) \cong \{ \tau \mid -\pi \leq \mathrm{Im}(\tau) < \pi \} \\
\rotatebox{90}{$\in$} & &  \rotatebox{90}{$\in$} \\
\tau  & \mapsto & \tau \mod 2 \pi \sqrt{-1} \mathbb{Z}.
\end{matrix}
\end{equation}
\end{proposition}
\begin{lemma}
\label{lem:criterion_for_existence_of_non-standard_component}
Let $S$ be a once-punctured torus, and consider the complex earthquake $\mathrm{Eq}_{\overline{\mathbb{H}} \cdot 1/0} (X_l) \subset P(S)$ for $l>0$.
Let $Z \in \mathrm{Eq}_{\overline{\mathbb{H}} \cdot 1/0} (X_l) \cap Q_{\mu}$ for some $\mu \in \mathcal{ML}_{\mathbb{Z}}(S)$.
If $\mu \notin \{ k \cdot 1/0  \mid  k = 0, 1, \cdots \}$, then $\mathrm{hol}(Z) \in \mathcal{QF}(l)$ is in a non-standard component.
\end{lemma}
\begin{proof}
Under the identification $\overline{\mathbb{H}} \cong \mathrm{Eq}_{\overline{\mathbb{H}} \cdot 1/0} (X_l) \subset P(S)$, 
the set of complex projective structures with Fuchsian holonomy corresponds to $\mathbb{R} + (2 \pi \sqrt{-1}) \cdot \mathbb{Z}_{\geq 0}$.
Moreover, for every $k \in \mathbb{Z}_{\geq 0}$, $\mathbb{R} + 2 \pi k \sqrt{-1}$ belongs to $Q_{k \cdot 1/0}$.

Let $U$ be the component of $\mathrm{Eq}_{\overline{\mathbb{H}} \cdot 1/0} (X_l) \cap Q_{\mu}$ containing $Z$.
If $\mu \notin \mathbb{Z}_{\geq 0} \cdot 1/0$, 
then $U \subset \{ \tau \in \overline{\mathbb{H}} \mid 2 (m-1) \pi < \mathrm{Im} (\tau) < 2 m \pi \}$ for some $m \in \mathbb{N}$.
Since the holonomy map has the form (\ref{eq:complex_earthquake_to_linear_slice}) on $\mathrm{Eq}_{\overline{\mathbb{H}} \cdot 1/0} (X_l)$, 
$\mathrm{hol}(U)$ is a component of $\mathcal{QF}(l)$, and $\mathrm{hol}(U)$ does not contain any Fuchsian representation.
Thus $\mathrm{hol}(U)$ is a non-standard component containing $Z$.
\end{proof}
The converse of Lemma \ref{lem:criterion_for_existence_of_non-standard_component} is true, but we need a little longer argument.
We postpone the proof of the next proposition to \S \ref{sec:pleated_surfaces}, since we will not use later.
\begin{proposition}
\label{prop:converse_to_the_criterion_lemma}
Let $S$ be a once-punctured torus, and consider the complex earthquake $\mathrm{Eq}_{\overline{\mathbb{H}} \cdot 1/0} (X_l) \subset P(S)$ for $l>0$.
Let $Z \in \mathrm{Eq}_{\overline{\mathbb{H}} \cdot 1/0} (X_l) \cap Q_{\mu}$ with $\mu \in \mathcal{ML}_{\mathbb{Z}}(S)$.
Then $\mathrm{hol}(Z) \in \mathcal{QF}(l)$ is in a non-standard component if and only if $\mu \notin \mathbb{Z}_{\geq 0} \cdot 1/0$.
\end{proposition}

By Lemma \ref{lem:criterion_for_existence_of_non-standard_component}, the non-emptiness of $\mathrm{Eq}_{\overline{\mathbb{H}} \cdot 1/0} (X_l) \cap Q_{\mu}$ for $\mu \notin \mathbb{Z}_{\geq 0} \cdot 1/0$ implies the existence of a non-standard component. 
The result of Komori and Yamashita \cite{komori-yamashita} (based on Otal's work \cite{otal}) implies if $l$ is sufficiently small, $\mathrm{Eq}_{\overline{\mathbb{H}} \cdot 1/0} (X_l) \cap Q_{\mu}$ is empty. 

\subsection{Second proof of Theorem \ref{komori_yamashita_theorem}}
\label{subsec:second_proof}
\begin{proof} 
Fix $k \in \mathbb{N}$ and $X \in \mathcal{T}(S)$.
We take simple closed curves $\alpha = 1/0$ and $\beta = 0/1$.
Consider the following sequence of complex projective structures in Thurston coordinates
\[
(\frac{2 \pi k}{n} \cdot D_{\beta}^{n} \alpha, X)
\]
where $D_{\beta}$ is the right Dehn twist along $\beta$.
This converges to $(2\pi k \beta, X)$ as $n \to \infty$, which is in $Q_{k \cdot \beta}$ since it is obtained from $X$ by $2\pi k$-grafting along $\beta$.
Since $Q_\beta$ is open, there exists $N$ such that $(\frac{2 \pi k}{n} \cdot D_{\beta}^{n} \alpha, X) \in Q_{k \cdot \beta}$ for all $n \geq N$.
Apply $D_\beta^{-n}$, we have 
\[
(\frac{2 \pi k}{n} \cdot \alpha, \, D_{\beta}^{-n} (X)) \in Q_{k \cdot D^n_{\beta} \beta} = Q_{k \cdot \beta}
\] 
for any $n \geq N$.
On the other hand, if we let $l = l_\alpha(D_\beta^{-n}(X))$, we have 
\[
(\frac{2 \pi k}{n} \cdot \alpha, \, D_{\beta}^{-n} (X)) \in \mathrm{Eq}_{\overline{\mathbb{H}} \cdot \alpha} (X_l).
\]
Since $\beta \notin \mathbb{Z}_{\geq 0} \cdot \alpha$, $\mathcal{QF}(l)$ has a non-standard component by Lemma \ref{lem:criterion_for_existence_of_non-standard_component}.
\end{proof}

\begin{remark}
The arguments above work even if we replace $\beta= 0/1$ with $p/1 \, (p \in \mathbb{Z})$.
We can take $l$ so that $Q_{1 \cdot p/1}$ intersects with $\mathrm{Eq}_{\overline{\mathbb{H}} \cdot 1/0} (X_l)$ for any $p \in \mathbb{Z}$.
This implies that there are infinitely many non-standard components in $\mathcal{QF}(l)$, 
although this follows immediately from Corollary \ref{cor:infinitely_many_by_dehn_twists}.
Since $p/1 \, (p \in \mathbb{Z})$ are related by Dehn twists along $1/0$, these components are the same after taking the quotient by the action of Dehn twists along $1/0$.

Furthermore, if we take sufficiently large $l$ so that $Q_{j \cdot 1/0 }$ intersects with $\mathrm{Eq}_{\overline{\mathbb{H}} \cdot 1/0} (X_l)$ for all $j = 1, 2, \dots, k$, 
$\mathcal{QF}(l)$ has more than $k$ components even after taking the quotient by the action of Dehn twists along $1/0$.
We remark that the non-locally connectivity shown by Bromberg \cite{bromberg} implies that $\mathcal{QF}(l)$ may have infinitely many components in the quotient.
\end{remark}

\subsection{Generalization}
Since the complex earthquake (\ref{eq:complex_earthquake}) and Thurston coordinates are defined for any hyperbolic surface,  
Lemma \ref{lem:criterion_for_existence_of_non-standard_component} and the construction in \S \ref{subsec:second_proof} can be generalized 
for general hyperbolic surfaces.
\begin{proposition}
Let $X$ be a hyperbolic surface and $\gamma$ a simple closed geodesic on $X$.
If $l_\gamma(X)$ is sufficiently large, the complex earthquake $\mathrm{Eq}_{\, \overline{\mathbb{H}} \cdot \gamma}(X)$ has a non-empty intersection with
$Q_{\mu}$ for some $\mu \in \mathcal{ML}_{\mathbb{Z}}$ but $\mu \notin \mathbb{Z}_{\geq 0} \cdot \gamma$.
\end{proposition}

\section{Pleated surfaces associated to real linear slices}
\label{sec:pleated_surfaces}
For every representation $\rho$ in a linear slice $X_{1/0}(l)$, 
there exists a pleated surface with pleating locus $1/0$ whose holonomy is $\rho$ up to conjugation. 
If the bending angle is small compared to $l$, this is realized as a convex core boundary.
The existence of a non-standard component is related to the existence of a pleated surface with pleating locus $1/0$ but not realized as a convex core boundary.
This fact was already observed in \cite{komori-yamashita}, but here we explain it for the proof of Proposition \ref{prop:converse_to_the_criterion_lemma}.

\subsection{Pleated surfaces}
We recall basic properties of (abstract) pleated surfaces. 
In this generality, we refer to \cite{bonahon}.

Let $S$ be a hyperbolic surface and $\lambda$ a geodesic lamination on $S$.
We regard the universal cover of $S$ with the hyperbolic plane $\mathbb{H}^2$, and let $\widetilde{\lambda}$ be a lift of $\lambda$ to $\mathbb{H}^2$.
A \emph{ pleated surface} of the pleating locus $\lambda$ is a pair $f = (\widetilde{f}, \rho)$ where 
$\rho : \pi_1(S) \to \PSLC$ is a representation and $\widetilde{f} : \mathbb{H}^2 \to \mathbb{H}^3$ is a $\rho$-equivariant 
map which sends each component of the complement $\mathbb{H}^2 \setminus \widetilde{\lambda}$ to totally geodesic surfaces.
We call $\rho$ its \emph{holonomy}, and  $\widetilde{f}$ its \emph{developing map} respectively. 
For example, the convex core boundaries of a quasi-Fuchsian representation are pleated surfaces.
But in general, $\rho$ is not assumed to be discrete.

\begin{definition}
We say that a pleated surface is \emph{convex} if $\widetilde{f}(\mathbb{H}^2)$ is the boundary of a convex subset of $\mathbb{H}^3$, 
and \emph{locally convex} if each point $p \in \mathbb{H}^2$ has a neighborhood $U$ such that $\widetilde{f}(U)$ is a part of the boundary of a convex subset of $\mathbb{H}^3$.
\end{definition}

Clearly, a convex pleated surface is locally convex. 
If a pleated surface is given by a convex core boundary of a quasi-Fuchsian representation, it is clearly a convex pleated surface. 

For a hyperbolic surface $X$ and a measured lamination $\mu$ on $X$, 
we can construct a developing map $\mathbb{H}^2 \to \mathbb{H}^3$ by bending $\mathbb{H}^2$ along the support of $\widetilde{\mu}$ according to its transverse measure.
This gives a locally convex pleated surface, but its holonomy is not discrete in general.
Conversely, a locally convex pleated surface with pleating locus $\lambda$ defines a signed transverse measure on $\lambda$ from the local convex structure.
If the transverse measure is positive, this gives a measured lamination.
(For non locally convex pleated surfaces, we need transverse H\"{o}lder distributions developed by Bonahon, instead of transverse measures.
In fact, Bonahon showed in \cite{bonahon} that the set of all abstract pleated surfaces of pleating locus $\lambda$ is parametrized by the Teichm\"{u}ller space of $S$ and 
the space of H\"{o}lder distributions for $\lambda$ with values in $\mathbb{R}/2\pi\mathbb{Z}$.)

\subsection{Real length curves}
From now on, we suppose that $S$ is a once-punctured torus.
Let $C_a$ be the simple closed curve corresponding to $1/0 \in \widehat{\mathbb{Q}}$ as in \S \ref{sec:character_variety}.
For $\rho \in X(l) = X_{1/0}(l)$, we consider the restriction $\rho |_{\pi_1(S \setminus C_a)}$.
We showed in the proof of Proposition \ref{prop:complexFN} that $\rho|_{\pi_1(S \setminus C_a)}$ is irreducible.
Since $S \setminus C_a$ is a three-holed sphere, $\rho |_{\pi_1(S \setminus C_a)}$ is completely determined up to conjugation by the traces of the holonomies along three boundary curves, 
thus $\rho |_{\pi_1(S \setminus C_a)}$ does not depend on $\rho \in X(l)$.
Since $X(l)$ contains a Fuchsian representation, $\rho |_{\pi_1(S \setminus C_a)}$ is Fuchsian.

Therefore if $\rho \in X(l)$, it can be realized as the holonomy of a pleated surface with pleating locus $1/0$.
This is locally convex but not convex in general. 
Suppose $\rho \in \mathcal{QF}(l)$, this pleated surface is convex if and only if 
it is realized as a convex core boundary, in other words, one of $[pl^{\pm}]$ is $1/0$ in $\mathcal{PML}(S)$.
Corollary \ref{cor:pleating_loci} can be rephrased as follows.
\begin{proposition}
\label{prop:belongs_to_BM}
Suppose $\rho \in \mathcal{QF}(l)$.
The pleated surface associated to $\rho$ with pleating locus $1/0$ is convex if and only if $\rho$ is in the standard component $BM$. 
\end{proposition}

\begin{proof}[Proof of Proposition \ref{prop:converse_to_the_criterion_lemma}]
By Lemma \ref{lem:criterion_for_existence_of_non-standard_component}, 
we only need to show that if $Z$ is in $\mathrm{Eq}_{\overline{\mathbb{H}} \cdot 1/0} (X_l) \cap Q_{k \cdot 1/0}$ $(k \in \mathbb{Z}_{\geq 0})$, 
then $\mathrm{hol}(Z)$ is in the standard component $BM$.
Assume that $Z \in \mathrm{Eq}_{\overline{\mathbb{H}} \cdot 1/0} (X_l) \cap Q_{k \cdot 1/0}$ $(k \in \mathbb{Z}_{\geq 0})$.
Now we can write $Z = \mathrm{Gr}_{b \cdot 1/0}( \mathrm{tw}_{t \cdot 1/0}(X_l))$ by $b \in \mathbb{R}_{\geq 0}$ and $t \in \mathbb{R}$.
Let $Z'$ be the complex projective structure obtained from $Z$ by removing $2 \pi$-annuli along $1/0$ as many as possible.
Then $Z' = \mathrm{Gr}_{b' \cdot 1/0}( \mathrm{tw}_{t \cdot 1/0}(X_l))$ where $0 \leq b' < 2 \pi$ and $Z'$ is in $Q_0$ or $Q_{1 \cdot 1/0}$.
If $Z' \in Q_0$, the injective developing map gives a convex pleated surface by the convex hull construction.
By Proposition \ref{prop:belongs_to_BM}, $\mathrm{hol}(Z) = \mathrm{hol}(Z')$ is in $BM$.
If $Z' \in Q_{1 \cdot 1/0}$, we consider $Z''= \mathrm{Gr}_{(2 \pi - b') \cdot 1/0}( \mathrm{tw}_{t \cdot 1/0}(X_l))$ 
whose holonomy is the complex conjugate of $\mathrm{hol}(Z')$ by (\ref{eq:complex_earthquake_to_linear_slice}), 
and $Z'' \in Q_0$ thus $\mathrm{hol}(Z'') \in BM$.
Since $BM$ is symmetric with respect to the real line, $\mathrm{hol}(Z') = \mathrm{hol}(Z)$ is also in $BM$.
\end{proof}

\section{Pictures}
\begin{figure}[ht]
\begin{center}
  \begin{minipage}{170pt}
    \begin{center}
      \includegraphics[width=160pt]{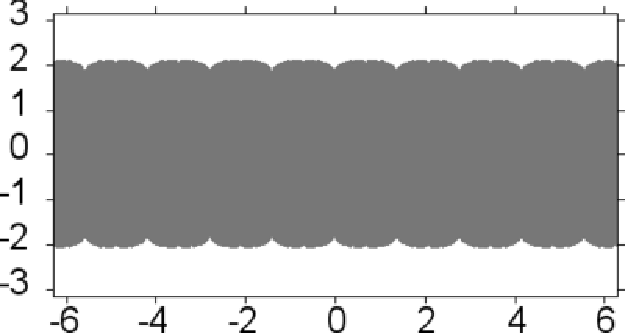}\\
      $l = 1.39$ $(\tr = 2.50\dots )$ 
    \end{center}
  \end{minipage}
  \begin{minipage}{170pt}
    \begin{center}
      \includegraphics[width=160pt]{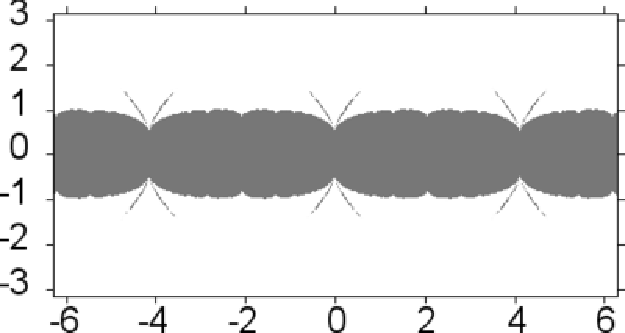}\\
      $l = 4.13$ $(\tr = 8.0\dots )$ 
    \end{center}
  \end{minipage}

  \vspace{20pt}
  \begin{minipage}{330pt}
    \begin{center}
      \includegraphics[width=320pt]{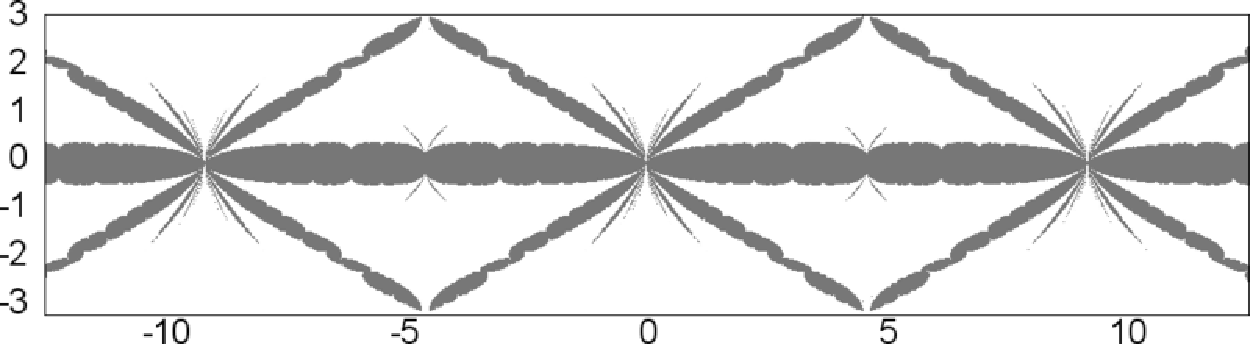}\\
      $l = 9.21$ $(\tr = 99.9\dots )$ 
    \end{center}
  \end{minipage}

  \vspace{20pt}
  \begin{minipage}{330pt}
    \begin{center}
      \includegraphics[width=320pt]{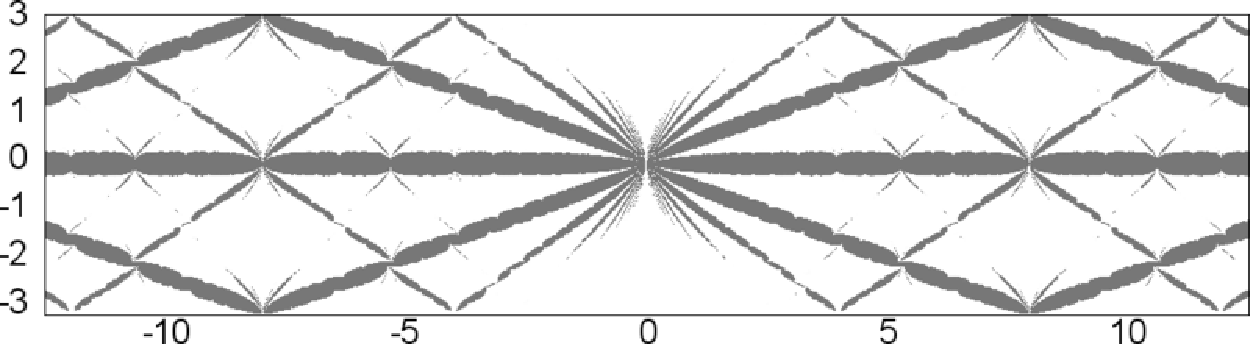}\\
      $l = 16.0$
    \end{center}
  \end{minipage}
\end{center}
\caption{Pictures of $\mathcal{QF}(l)$ (shaded regions).}
\label{fig:slices}
\end{figure}

Some pictures of $\mathcal{QF}(l)$ are shown in Figure \ref{fig:slices}.
These are drawn by a program plotting the points satisfying Bowditch's conditions \cite{bowditch}, 
which are conjectured and experimentally well confirmed to be equivalent to the condition that the corresponding representation is quasi-Fuchsian. 
Compare with Figures 1, 2, 3 in \cite{komori-yamashita}, which are plotted in the trace coordinates for $\tr = 2.5$, $\tr = 8$, $\tr = 100$ cases, respectively.
The asymptotic self-similarity in \cite[\S 7]{komori-yamashita} is nothing but translation symmetry in our coordinates.

Some developing maps are drawn in Figure \ref{fig:developing_maps}, 
but the last one ($\tau = 0.40 + 0.70 \sqrt{-1}$) is a partial picture since the developing map is not an embedding.
In these pictures, the lifts of the grafted annulus can be seen as white crescent regions.
When $\tau = 0.40 + 0.70 \sqrt{-1}$, we can observe that the limit set traverses these lifts.
This implies that the curve with integral weight appeared in Goldman's classification is not homotopic into the grafted annulus. 

\begin{figure}[ht]
\begin{center}
  \begin{minipage}{170pt}
    \begin{center}
      \includegraphics[width=150pt]{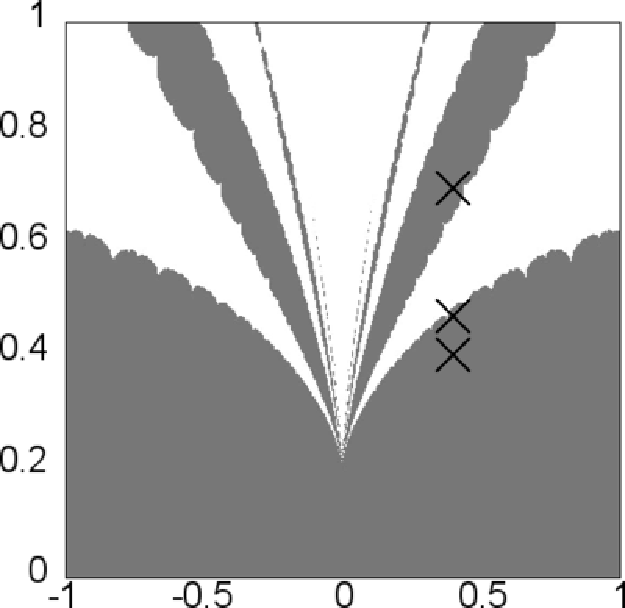}\\
      $l=6.00$
    \end{center}
  \end{minipage}
  \begin{minipage}{170pt}
    \begin{center}
      \includegraphics[width=170pt]{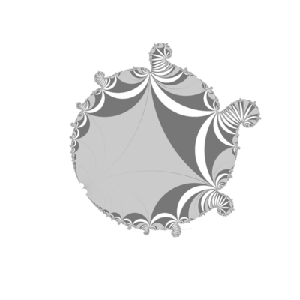}\\
      $\tau = 0.40 + 0.40 \sqrt{-1}$
    \end{center}
  \end{minipage}
  \begin{minipage}{170pt}
    \begin{center}
      \includegraphics[width=170pt]{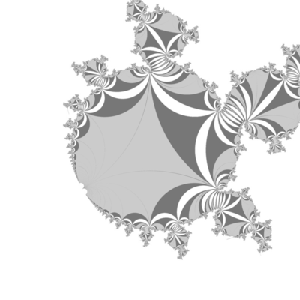}\\
      $\tau = 0.40 + 0.47 \sqrt{-1}$
    \end{center}
  \end{minipage}
  \begin{minipage}{170pt}
    \begin{center}
      \includegraphics[width=170pt]{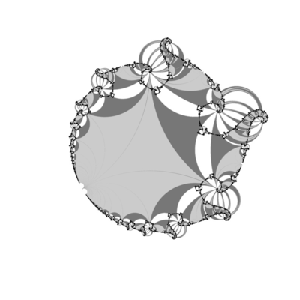}\\
      $\tau = 0.40 + 0.70 \sqrt{-1}$
    \end{center}
  \end{minipage}
\end{center}
\caption{$\mathcal{QF}(l)$ for $l=6.00$ near the origin and developing maps.}
\label{fig:developing_maps}
\end{figure}

\end{document}